\documentclass[11pt]{amsart}    
\usepackage{latexsym, amsmath, amsthm, amssymb, setspace, verbatim}
\usepackage[all]{xy}
\usepackage[T1]{fontenc} 
\usepackage{lmodern}
\usepackage[utf8]{inputenc} 
\usepackage{hyperref}
\usepackage{enumitem}

\title[Nuclear dimension of purely infinite $\mathrm{C}^*$-algebras]{ On the nuclear dimension of strongly purely infinite $\mathrm{C}^*$-algebras }
\author{Gábor Szabó}
\address{Fraser Noble Building, Institute of Mathematics, University of Aberdeen, \linebreak \text{}\hspace{3mm} Aberdeen AB24 3UE, Scotland, UK.}
\email{gabor.szabo@abdn.ac.uk}
\thanks{\emph{Supported by:} SFB 878 \emph{Groups, Geometry and Actions} and EPSRC grant EP/N00874X/1.}

\begin{document}

\renewcommand\matrix[1]{\left(\begin{array}{*{10}{c}} #1 \end{array}\right)}  
\newcommand\set[1]{\left\{#1\right\}}  
\newcommand\mset[1]{\left\{\!\!\left\{#1\right\}\!\!\right\}}

\newcommand{\IA}[0]{\mathbb{A}} \newcommand{\IB}[0]{\mathbb{B}}
\newcommand{\IC}[0]{\mathbb{C}} \newcommand{\ID}[0]{\mathbb{D}}
\newcommand{\IE}[0]{\mathbb{E}} \newcommand{\IF}[0]{\mathbb{F}}
\newcommand{\IG}[0]{\mathbb{G}} \newcommand{\IH}[0]{\mathbb{H}}
\newcommand{\II}[0]{\mathbb{I}} \renewcommand{\IJ}[0]{\mathbb{J}}
\newcommand{\IK}[0]{\mathbb{K}} \newcommand{\IL}[0]{\mathbb{L}}
\newcommand{\IM}[0]{\mathbb{M}} \newcommand{\IN}[0]{\mathbb{N}}
\newcommand{\IO}[0]{\mathbb{O}} \newcommand{\IP}[0]{\mathbb{P}}
\newcommand{\IQ}[0]{\mathbb{Q}} \newcommand{\IR}[0]{\mathbb{R}}
\newcommand{\IS}[0]{\mathbb{S}} \newcommand{\IT}[0]{\mathbb{T}}
\newcommand{\IU}[0]{\mathbb{U}} \newcommand{\IV}[0]{\mathbb{V}}
\newcommand{\IW}[0]{\mathbb{W}} \newcommand{\IX}[0]{\mathbb{X}}
\newcommand{\IY}[0]{\mathbb{Y}} \newcommand{\IZ}[0]{\mathbb{Z}}

\newcommand{\CA}[0]{\mathcal{A}} \newcommand{\CB}[0]{\mathcal{B}}
\newcommand{\CC}[0]{\mathcal{C}} \newcommand{\CD}[0]{\mathcal{D}}
\newcommand{\CE}[0]{\mathcal{E}} \newcommand{\CF}[0]{\mathcal{F}}
\newcommand{\CG}[0]{\mathcal{G}} \newcommand{\CH}[0]{\mathcal{H}}
\newcommand{\CI}[0]{\mathcal{I}} \newcommand{\CJ}[0]{\mathcal{J}}
\newcommand{\CK}[0]{\mathcal{K}} \newcommand{\CL}[0]{\mathcal{L}}
\newcommand{\CM}[0]{\mathcal{M}} \newcommand{\CN}[0]{\mathcal{N}}
\newcommand{\CO}[0]{\mathcal{O}} \newcommand{\CP}[0]{\mathcal{P}}
\newcommand{\CQ}[0]{\mathcal{Q}} \newcommand{\CR}[0]{\mathcal{R}}
\newcommand{\CS}[0]{\mathcal{S}} \newcommand{\CT}[0]{\mathcal{T}}
\newcommand{\CU}[0]{\mathcal{U}} \newcommand{\CV}[0]{\mathcal{V}}
\newcommand{\CW}[0]{\mathcal{W}} \newcommand{\CX}[0]{\mathcal{X}}
\newcommand{\CY}[0]{\mathcal{Y}} \newcommand{\CZ}[0]{\mathcal{Z}}

\newcommand{\FA}[0]{\mathfrak{A}} \newcommand{\FB}[0]{\mathfrak{B}}
\newcommand{\FC}[0]{\mathfrak{C}} \newcommand{\FD}[0]{\mathfrak{D}}
\newcommand{\FE}[0]{\mathfrak{E}} \newcommand{\FF}[0]{\mathfrak{F}}
\newcommand{\FG}[0]{\mathfrak{G}} \newcommand{\FH}[0]{\mathfrak{H}}
\newcommand{\FI}[0]{\mathfrak{I}} \newcommand{\FJ}[0]{\mathfrak{J}}
\newcommand{\FK}[0]{\mathfrak{K}} \newcommand{\FL}[0]{\mathfrak{L}}
\newcommand{\FM}[0]{\mathfrak{M}} \newcommand{\FN}[0]{\mathfrak{N}}
\newcommand{\FO}[0]{\mathfrak{O}} \newcommand{\FP}[0]{\mathfrak{P}}
\newcommand{\FQ}[0]{\mathfrak{Q}} \newcommand{\FR}[0]{\mathfrak{R}}
\newcommand{\FS}[0]{\mathfrak{S}} \newcommand{\FT}[0]{\mathfrak{T}}
\newcommand{\FU}[0]{\mathfrak{U}} \newcommand{\FV}[0]{\mathfrak{V}}
\newcommand{\FW}[0]{\mathfrak{W}} \newcommand{\FX}[0]{\mathfrak{X}}
\newcommand{\FY}[0]{\mathfrak{Y}} \newcommand{\FZ}[0]{\mathfrak{Z}}

\newcommand{\fc}[0]{\mathfrak{c}}

\newcommand{\Ra}[0]{\Rightarrow}
\newcommand{\La}[0]{\Leftarrow}
\newcommand{\LRa}[0]{\Leftrightarrow}

\renewcommand{\phi}[0]{\varphi}
\newcommand{\eps}[0]{\varepsilon}

\newcommand{\quer}[0]{\overline}
\newcommand{\uber}[0]{\choose}
\newcommand{\ord}[0]{\operatorname{ord}}		
\newcommand{\GL}[0]{\operatorname{GL}}
\newcommand{\supp}[0]{\operatorname{supp}}	
\newcommand{\id}[0]{\operatorname{id}}		
\newcommand{\Sp}[0]{\operatorname{Sp}}		
\newcommand{\eins}[0]{\mathbf{1}}			
\newcommand{\diag}[0]{\operatorname{diag}}
\newcommand{\auf}[1]{\quad\stackrel{#1}{\longrightarrow}\quad}
\newcommand{\prim}[0]{\operatorname{Prim}}
\newcommand{\ad}[0]{\operatorname{Ad}}
\newcommand{\ext}[0]{\operatorname{Ext}}
\newcommand{\ev}[0]{\operatorname{ev}}
\newcommand{\fin}[0]{{\subset\!\!\!\subset}}
\newcommand{\diam}[0]{\operatorname{diam}}
\newcommand{\Hom}[0]{\operatorname{Hom}}
\newcommand{\Aut}[0]{\operatorname{Aut}}
\newcommand{\del}[0]{\partial}
\newcommand{\dimnuc}[0]{\dim_{\mathrm{nuc}}}
\newcommand{\dr}[0]{\operatorname{dr}}
\newcommand{\dimrok}[0]{\dim_{\mathrm{Rok}}}
\newcommand{\dimrokcyc}[0]{\dim_{\mathrm{Rok}}^{\mathrm{cyc}}}
\newcommand{\dimrokcycc}[0]{\dim_{\mathrm{Rok}}^{\mathrm{cyc,c}}}
\newcommand{\dimnuceins}[0]{\dimnuc^{\!+1}}
\newcommand{\dreins}[0]{\dr^{\!+1}}
\newcommand{\dimrokeins}[0]{\dimrok^{\!+1}}
\newcommand{\reldimrok}[2]{\dimrok(#1~|~#2)}
\newcommand{\mdim}[0]{\operatorname{mdim}}
\newcommand*\onto{\ensuremath{\joinrel\relbar\joinrel\twoheadrightarrow}} 
\newcommand*\into{\ensuremath{\lhook\joinrel\relbar\joinrel\rightarrow}}  
\newcommand{\dst}[0]{\displaystyle}
\newcommand{\cstar}[0]{$\mathrm{C}^*$}
\newcommand{\dist}[0]{\operatorname{dist}}
\newcommand{\ue}[0]{{\approx_{\mathrm{u}}}}
\newcommand{\End}[0]{\operatorname{End}}
\newcommand{\Ell}[0]{\operatorname{Ell}}
\newcommand{\inv}[0]{\operatorname{Inv}}
\newcommand{\ann}[0]{\operatorname{Ann}}

\newtheorem{satz}{Satz}[section]		
\newtheorem{cor}[satz]{Corollary}
\newtheorem{lemma}[satz]{Lemma}
\newtheorem{prop}[satz]{Proposition}
\newtheorem{theorem}[satz]{Theorem}
\newtheorem*{theoreme}{Theorem}

\theoremstyle{definition}
\newtheorem{defi}[satz]{Definition}
\newtheorem*{defie}{Definition}
\newtheorem{defprop}[satz]{Definition \& Proposition}
\newtheorem{nota}[satz]{Notation}
\newtheorem*{notae}{Notation}
\newtheorem{rem}[satz]{Remark}
\newtheorem*{reme}{Remark}
\newtheorem{example}[satz]{Example}
\newtheorem{defnot}[satz]{Definition \& Notation}
\newtheorem{question}[satz]{Question}
\newtheorem*{questione}{Question}

\newenvironment{bew}{\begin{proof}[Proof]}{\end{proof}}

\begin{abstract} 
We show that separable, nuclear and strongly purely infinite \cstar-algebras have finite nuclear dimension. In fact, the value is at most three. This exploits a deep structural result of Kirchberg and R{\o}rdam on strongly purely infinite \cstar-algebras that are homotopic to zero in an ideal-system preserving way.
\end{abstract}

\maketitle


\setcounter{section}{-1}

\section{Introduction}
\noindent
Over the last decade, our understanding of simple, nuclear C*-algebras has advanced enormously. This has largely been spurred by the interplay between certain topological and algebraic regularity properties, such as finite topological dimension, tensorial absorption of suitable strongly self-absorbing \cstar-algebras and order completeness of homological invariants, see \cite{ElliottToms08} for an overview. Indeed, this reflects the Toms-Winter regularity conjecture \cite{ElliottToms08, Winter10, WinterZacharias10}, which has spawned an investigation of simple \cstar-algebras that has lead to deep results at a fast and furious pace. We refer to \cite{BBSTWW} and its excellent introduction for an overview of the current state-of-the-art concerning the Toms-Winter regularity conjecture. The most recent highlight of the Elliott classification program is certainly the combination of papers \cite{GongLinNiu15, ElliottNiu15, ElliottGongLinNiu15, TikuisisWhiteWinter15} (by many hands), which have completed the classification of separable, unital, simple \cstar-algebras with finite nuclear dimension and satisfying the UCT. See in particular \cite[Section 6]{TikuisisWhiteWinter15} for a comprehensive summary.

Apart from applications in classification, it now stands to reason (see \cite{TikuisisWinter14,  RobertTikuisis15, BarlakEndersMatuiSzaboWinter}) that the interplay between finite nuclear dimension and $\CZ$-stability goes beyond the simple case and should indeed extend to the realm of non-simple, nuclear \cstar-algebras that are sufficiently non-type I. In particular, it is tempting to conjecture that all separable, nuclear and $\CZ$-stable \cstar-algebras have finite nuclear dimension. Given that Kirchberg has established a deep classification result \cite{KirchbergC} for separable, nuclear, strongly purely infinite \cstar-algebras long ago, a natural step toward this conjecture is to handle the $\CO_\infty$-absorbing case first. Note that a separable, nuclear \cstar-algebra $A$ is strongly purely infinite if and only if $A\cong A\otimes\CO_\infty$ (see \cite[3.2]{TomsWinter07} and \cite[5.11(iii), 8.6]{KirchbergRordam02}), and if and only if $A\cong A\otimes\CZ$ and $A$ is traceless (see \cite[3.12]{Kirchberg04}).

In this note, we take this first step towards the above conjecture and verify that all separable, nuclear and $\CO_\infty$-stable \cstar-algebras have finite nuclear dimension, and in fact the nuclear dimension is at most three. In \cite{BarlakEndersMatuiSzaboWinter}, it was shown that a nuclear dimension estimate 
\[
\dimnuceins(A\otimes\CO_\infty)\leq 2\dimnuceins(A\otimes\CO_2)
\] 
holds for all \cstar-algebras $A$. As an application, one could see that separable, nuclear, $\CO_\infty$-absorbing $\CC(X)$-algebras with simple fibres have finite nuclear dimension. Extending this technique in the first section, we show that in fact, a more general nuclear dimension estimate 
\[
\dimnuceins(A\otimes\CO_\infty)\leq 2\dimnuceins(A\otimes B)
\]
holds for all \cstar-algebras $A$ and all non-zero, separable \cstar-algebras $B$ with $B\cong B\otimes\CO_\infty$. This will be an application of Voiculescu's observation \cite{Voiculescu91} that cones over \cstar-algebras are quasidiagonal, combined with a trick from \cite{BarlakEndersMatuiSzaboWinter} involving positive elements with full spectrum in simple, purely infinite \cstar-algebras.
In particular, the above estimate implies that $\CO_\infty$ has a dimension-reducing effect on separable, nuclear \cstar-algebras upon tensorial stabilization if and only if there is some non-zero, separable \cstar-algebra with this property. The deep stuctural results by Kirchberg and R{\o}rdam \cite{KirchbergRordam05} on \cstar-algebras homotopic to zero in an ideal-system preserving way then show that a certain purely infinite AH algebra constructed by R{\o}rdam in \cite{Rordam04} has such a dimension-reducing effect. This yields the main result.


\section{2-colored embeddings into purely infinite ultrapowers}

\noindent
The following is contained in the proof of \cite[3.3]{BarlakEndersMatuiSzaboWinter} and was originally observed by Winter:

\begin{lemma} \label{full spectrum pi}
Let $A$ be a unital, simple and purely infinite \cstar-algebra. Let $h\in A$ be a positive contraction with full spectrum $[0,1]$. For every $\eps>0$, there is some $n\in\IN$, numbers $0\leq \lambda_0\leq \lambda_1\leq\dots\leq\lambda_n \leq 1$ and pairwise orthogonal projections $p_0,\dots,p_n\in A$ such that $h=_\eps \sum_{j=0}^n \lambda_j p_j$. Moreover, all the projections $p_j$ may be chosen to represent the trivial class in $K_0(A)$.
\end{lemma}
\begin{proof}
Since $A$ has real rank zero by \cite[4.1.1]{Rordam}, the only thing left to show is that the projections may be chosen to be trivial in $K$-theory. Let $\eps>0$. Choose $n\in\IN$, numbers $0\leq \lambda_0\leq\dots\leq\lambda_n \leq 1$ and pairwise orthogonal projections $q_0,\dots,q_n\in A$ as above without the $K$-theory assumption.

By \cite[Section 1]{Cuntz81}, we can find a non-trivial projection $\tilde{q}_n\leq q_n$ with $[\tilde{q}_n]_0=0\in K_0(A)$, and set $p_n=\tilde{q}_n$. Now assume that $j\in\set{1,\dots,n}$ is a number for which we have already constructed projections $\tilde{q}_n,\dots,\tilde{q}_{n-j+1}$ and $p_n,\dots,p_{n-j+1}$. Apply \cite[Section 1]{Cuntz81} again to find a non-trivial projection $\tilde{q}_{n-j}\leq q_{n-j}$ such that $[\tilde{q}_{n-j}]_0=-[q_{n-j+1}-p_{n-j+1}]_0\in K_0(A)$, and set $p_{n-j} = \tilde{q}_{n-j}+q_{n-j+1}-p_{n-j+1}$. After $n$ steps, we have constructed pairwise orthogonal projections $p_0,\dots,p_n$, all being trivial in $K_0(A)$, such that
\[
\sum_{j=0}^n \lambda_n (q_n-p_n) = \lambda_0(q_0-\tilde{q}_0) + \sum_{j=1}^n (\lambda_j-\lambda_{j-1})(q_j-\tilde{q}_j).
\]
Since $h$ has full spectrum, we necessarily have $\lambda_0\leq\eps$ and $\lambda_j-\lambda_{j-1}\leq\eps$ for all $j\in\set{1,\dots,n}$. It follows that the above sum has norm at most $\eps$. Hence $h=_{2\eps} \sum_{j=0}^n \lambda_n p_n$.
\end{proof}

As a consequence, we get:

\begin{lemma} \label{full spectrum 1}
Let $A$ be a unital, simple and purely infinite \cstar-algebra. Let $h_0,h_1\in A$ be two positive contractions with full spectrum $[0,1]$. For every $\eps>0$, there is a unitary $u\in A$ with $h_1 =_\eps uh_0u^*$.
\end{lemma}
\begin{proof}
Let $\eps>0$. By \ref{full spectrum pi}, we can find numbers $k,n\in\IN$, numbers $0\leq\lambda_0\leq\dots\leq\lambda_k\leq 1$ and $0\leq\mu_0\leq\dots\leq\mu_n\leq 1$, pairwise orthogonal projections $p_0,\dots,p_k$ and $q_0,\dots,q_n$ in $A$ that are all trivial in $K_0(A)$, such that
\[
\sum_{i=0}^k \lambda_ip_i =_\eps h_0\quad\text{and}\quad \sum_{j=0}^n \mu_jq_j =_\eps h_1.
\]
By breaking up projections, if necessary, we may assume that $k=n$ and that $|\lambda_i-\mu_i|\leq\eps$ for all $i=0,\dots,n$. By breaking up the projections $p_0, q_0$, if necessary, we may moreover assume that $\sum_{i=0}^n p_i \neq \eins \neq \sum_{i=0}^n q_i$.

It follows from \cite[Section 1]{Cuntz81} that there exist partial isometries $v_0,\dots,v_n,v\in A$ with
\[
v^*v = \eins_A-\sum_{j=0}^n p_j,~ vv^*=\eins_A-\sum_{j=0}^n q_j
\]
and
\[
v_j^*v_j = p_j,~ v_jv_j^*=q_j\quad\text{for all}~j=0,\dots,n.
\]
The element $u=v+v_0+\dots+v_n\in A$ then defines a unitary. We have
\[
\begin{array}{clcll}
uh_0u^* &=_\eps& \dst u\Bigl(\sum_{j=0}^n \lambda_j p_j\Bigl)u^* 
&=& \dst\sum_{j=0}^n \lambda_j up_ju^* \\
&=& \dst\sum_{j=0}^n \lambda_j v_j\cdot v_j^*v_j\cdot v_j^*
&=& \dst\sum_{j=0}^n \lambda_j q_j \\
&=_\eps& \dst\sum_{j=0}^n \mu_j q_j 
&=_\eps& h_1.
\end{array} 
\]
This finishes the proof.
\end{proof}

\begin{lemma} \label{full spectrum}
Let $A$ be a unital, simple and purely infinite \cstar-algebra. Let $\omega\in\beta\IN\setminus\IN$ be a free ultrafilter. Let $h_0,h_1\in A_\omega$ be two positive contractions with full spectrum $[0,1]$. Then $h_0$ and $h_1$ are unitarily equivalent.
\end{lemma}
\begin{proof}
Since $A_\omega$ is again simple and purely infinite, this follows directly from \ref{full spectrum 1} and a standard reindexation argument.
\end{proof}

The main technical observation of this section is the following:

\begin{prop} \label{2-color-embedding}
Let $A$ be a separable \cstar-algebra with a positive element $e\in A$ of norm 1. Let $\omega\in\beta\IN\setminus\IN$ be a free ultrafilter. Then there exist two c.p.c.~order zero maps $\phi_0,\phi_1: A\to (\CO_\infty)_\omega$ with $\phi_0(e)+\phi_1(e)=\eins$.
\end{prop}
\begin{proof}
By a result of Voiculescu \cite{Voiculescu91}, the cone over $A$ is quasidiagonal. Hence we can find a $*$-monomorphism $\kappa: \CC_0\big( (0,1], A \big) \to\CQ_\omega$, where $\CQ$ is the universal UHF algebra. By choosing some (necessarily non-unital) embedding $\mu: \CQ\to\CO_\infty$\footnote{For example, this can be a composition of some non-unital embedding $\CO_2\into\CO_\infty$ and some unital embedding $\CQ\into\CO_2$, which exists by \cite[2.8]{KirchbergPhillips00}.}, the composition $\psi_0=\mu_\omega\circ\kappa: \CC_0\big( (0,1], A \big) \to (\CO_\infty)_\omega$ is a $*$-monomorphism. The positive element $h=\psi_0(\id_{[0,1]}\otimes e)\in (\CO_\infty)_\omega$ then has full spectrum $[0,1]$.
 
Applying \ref{full spectrum}, we find a unitary $u\in (\CO_\infty)_\omega$ with $uhu^* = \eins-h$. Then $\psi_1 = \ad(u)\circ\psi_0: \CC_0\big( (0,1], A \big) \to (\CO_\infty)_\omega$ is also a $*$-monomorphism.
This gives rise to two c.p.c.~order zero maps $\phi_0,\phi_1: A\to (\CO_\infty)_\omega$ via $\phi_i(x)=\psi_i(\id_{[0,1]}\otimes x)$ for all $x\in A$ and $i=0,1$. We then have
\[
\phi_0(e)+\phi_1(e) = h + uhu^* = \eins,
\]
as required.
\end{proof}

As a consequence, we get a general nuclear dimension estimate for $\CO_\infty$-stable \cstar-algebras that generalizes \cite[3.3]{BarlakEndersMatuiSzaboWinter}:

\begin{cor} \label{Oinf estimate}
Let $A$ be a \cstar-algebra and $B$ a non-zero, separable, $\CO_\infty$-stable \cstar-algebra. Then
\[
\dimnuceins(A\otimes\CO_\infty)\leq 2\dimnuceins(A\otimes B).
\]
\end{cor}
\begin{proof}
Let $d=\dimnuc(A\otimes B)$.
Since $B$ is assumed to be $\CO_\infty$-absorbing, we may also assume without loss of generality that $A$ is $\CO_\infty$-absorbing. We may certainly also assume that $A$ is nuclear, as there is otherwise nothing to show. By \cite[2.5, 2.6]{TikuisisWinter14}, it suffices to find an estimate for the nuclear dimension of the embedding $(\id_A\otimes\eins)_\omega: A\to (A\otimes\CO_\infty)_\omega$.

So let $F\fin A$ and $\eps>0$ be arbitrary. As $B$ is non-zero, we can find some positive contraction $e\in B$ of norm 1. Apply \ref{2-color-embedding} to find two c.p.c.~order zero maps $\phi_0,\phi_1: B\to (\CO_\infty)_\omega$ with $\phi_0(e)+\phi_1(e)=\eins$.
By the definition of $d$, we can find a finite-dimensional \cstar-algebra $\CF$, a c.p.c.~map $\mu: A\otimes B\to\CF$ and c.p.c.~order zero maps $\kappa^{(0)},\dots,\kappa^{(d)}: \CF\to A\otimes B$ with
\[
x\otimes e =_\eps \sum_{j=0}^d \kappa^{(j)}\circ\mu(x\otimes e)\quad\text{for all}~x\in F.
\]
Define the c.p.c.~map $\Psi: A\to\CF$ via $\Psi(x)=\mu(x\otimes e)$. Moreover, define for all $i=0,1$ and $j=0,\dots,d$ the c.p.c.~order zero map $\Phi^{(i,j)}: \CF\to (A\otimes\CO_\infty)_\omega$ via $\Phi^{(i,j)} = (\id_A\otimes\phi_i)\circ\kappa^{(j)}$. Then we have
\[
\begin{array}{ccl}
\dst\sum_{i=0,1}\sum_{j=0}^d \Phi^{(i,j)}\circ\Psi(x) &=& \dst\sum_{i=0,1} (\id_A\otimes\phi_i)\left( \sum_{j=0}^d \kappa^{(j)}\circ\mu(x\otimes e) \right) \\\\
&=_{2\eps}& \dst\sum_{i=0,1} x\otimes\phi_i(e) ~=~ x\otimes\eins
\end{array}
\]
for all $x\in F$. This yields a $2(d+1)$-colored decomposition of the map $(\id_A\otimes\eins)_\omega: A\to (A\otimes\CO_\infty)_\omega$ through finite-dimensional \cstar-algebras and finishes the proof.
\end{proof}


\section{$\CO_\infty$-stable \cstar-algebras have finite nuclear dimension}

\noindent
In this section, we prove our main result. This will be a consequence of the estimate \ref{Oinf estimate}, together with a deep structural result of Kirchberg and R{\o}rdam from \cite{KirchbergRordam05}. Let us now recall some of these results:

\begin{defi}[see \cite{Rordam04}] Let $\set{t_n}_{n\in\IN}\subset [0,1)$ be a dense sequence. For every $n$, define the $*$-homomorphism
\[
\phi_n: \CC_0\bigl( [0,1), M_{2^n} \bigl) \to \CC_0\bigl( [0,1), M_{2^{n+1}} \bigl)
\]
via
\[\phi_n(f)(t) = \diag\Big( f(t), f\big( \max(t, t_n) \big) \Big)\quad\text{for all}~t\in [0,1).
\]
Set $\dst \CA_{[0,1]} = \lim_{\longrightarrow} \set{ \CC\bigl( [0,1), M_{2^n} \bigl), \phi_n }$.
\end{defi}

In \cite{KirchbergRordam05}, Kirchberg and R{\o}rdam have shown a deep structural result for separable, nuclear, strongly purely infinite \cstar-algebras that are homotopic to zero in an ideal-system preserving way. An interesting consequence of this observation is that tensoring a given nuclear \cstar-algebra with $\CA_{[0,1]}$ yields an $\mathrm{AH}_0$ algebra with low dimension:

\begin{theorem}[see {\cite[5.12, 6.2]{KirchbergRordam05}}] \label{KiRo}
For every non-zero, separable, nuclear \cstar-algebra $B$, the \cstar-algebra $B\otimes\CA_{[0,1]}$ is an $\mathrm{AH}_0$ algebra of topological dimension one. In particular, the decomposition rank of $B\otimes\CA_{[0,1]}$ is one.
\end{theorem}

Combining this structural result with our nuclear dimension estimate from the previous section, we obtain our main result:

\begin{theorem} \label{main result}
Let $A$ be a separable, nuclear \cstar-algebra. Then the nuclear dimension of $A\otimes\CO_\infty$ is at most three.
\end{theorem}
\begin{proof}
The \cstar-algebra $\CA_{[0,1]}$ is $\CO_\infty$-stable by \cite[5.3]{Rordam04}.
So combining \ref{Oinf estimate} and \ref{KiRo}, we get
\[
\dimnuceins(A\otimes\CO_\infty) \leq 2\dimnuceins(A\otimes\CA_{[0,1]})\leq 4.
\]
This shows our claim.
\end{proof}

\begin{rem}
At first glance, it might appear that \ref{KiRo} (and therefore \ref{main result}) depends on Kirchberg's classification \cite{KirchbergC} of non-simple, strongly purely infinite \cstar-algebras. However, Kirchberg's classification theorem is only used in \cite{KirchbergRordam05} to deduce that $\CA_{[0,1]}$ (or more generally any separable, nuclear, strongly purely infinite \cstar-algebras homotopic to zero in an ideal system-preserving way) is in fact $\CO_2$-absorbing. While this is an important ingredient in verifying the $\mathrm{AH}_0$-structure of the \cstar-algebras in question, one can refrain from appealing to classification by weakening the statement in \ref{KiRo} to $\dr(B\otimes\CA_{[0,1]}\otimes\CO_2)=1$ for all non-zero, separable, nuclear \cstar-algebras $B$. In particular, this weaker statement can be proved exclusively with the methods from \cite{KirchbergRordam05}, without classification, and is sufficient to deduce our main result \ref{main result}.
\end{rem}


\bibliographystyle{gabor}
\bibliography{master}

\end{document}